\documentclass[dvipdfmx, a4paper, 11pt]{amsart}
\usepackage{amsmath,amsthm,amssymb,amsthm, amsfonts, mathrsfs}
\usepackage{tikz}
\usetikzlibrary{intersections,calc,arrows.meta}
\makeatletter
 \def\@makefnmark{%
 \leavevmode
 \raise.9ex\hbox{\check@mathfonts
 \fontsize\sf@size\z@\normalfont%
 \@thefnmark}%
 }
 \makeatother
\newcommand\diam{\operatorname{diam}}
\newcommand{\HH}{\mathbb{H}}
\newcommand{\N}{\mathbb{N}}

\newcommand{\R}{\mathbb{R}}

\newcommand{\LL}{\mathcal{L}}
\newcommand{\F}{\mathscr{F}}
\newcommand{\G}{\mathscr{G}}
\newcommand{\U}{\mathcal{U}}
\newcommand{\SO}{\mathcal{SO}}

\newcommand{\aaa}{\mathfrak{a}}
\newcommand{\bbb}{\mathfrak{b}}
\newcommand{\ccc}{\mathfrak{c}}
\theoremstyle{plain}
\newtheorem{theorem}{Theorem}[section]
\newtheorem{lemma}[theorem]{Lemma}
\newtheorem{proposition}[theorem]{Proposition}
\newtheorem{corollary}[theorem]{Corollary}
\theoremstyle{definition}
\newtheorem{definition}[theorem]{Definition}

\theoremstyle{remark}

\newcounter{cn}
\title[Homomorphisms of the lattice of slowly oscillating functions]{Homomorphisms of the lattice of slowly oscillating functions on the half-line}
\author{Yutaka Iwamoto
}
\address{Faculty of Fundamental Science, National Institute of Technology (KOSEN), Niihama College,
Niihama, 792-8580, Japan}
\email{y.iwamoto@niihama-nct.ac.jp}
\subjclass[2010]{46E05, 54C35}
\keywords{Slowly oscillating functions, uniformly continuous functions, lattice homomorphisms, Higson compactification, Samuel-Smirnov compactification}
\begin{document}

\begin{abstract}
We study the space $H(\SO)$ of all homomorphisms of the vector lattice of all slowly oscillating functions on the half-line $\HH=[0,\infty)$.
In contrast to the case of homomorphisms of uniformly continuous functions, it is shown that a homomorphism in $H(\SO)$ which maps the unit to zero must be the zero-homomorphism.
Consequently, we show that the space $H(\SO)$ without the zero-homomorphism is homeomorphic to $\HH\times (0, \infty)$.
By describing a neighborhood base of the zero-homomorphism, we show that $H(\SO)$ is homeomorphic to the space $\HH\times (0, \infty)$ with one point added.
\end{abstract}

\maketitle

\section{Introduction}
The aim of this note is to describe the real-valued homomorphisms of the vector lattice of all slowly oscillating functions on the half-line $\HH=[0, \infty)$.

Slowly oscillating functions are used to define Higson compactifications \cite{Keesling}
 and are functions that appear frequently in coarse geometry.
By analyzing slowly oscillating functions on $\HH$, it follows that its Higson corona $\nu\HH$ is a non-metrizable indecomposable continuum.
Although this fact is topologically interesting in its own right, in the context of  geometric group theory, it is applied to characterize the number of ends of finitely generated groups by whether the components of its Higson corona are decomposable or not \cite{Iwa2}.

Let $\U$ be the vector lattice of all uniformly continuous functions on $\HH$ and $\U^{\ast}$
 the sublattice of bounded functions.
In \cite{CS}, F\'elix Cabello S\'anchez analyzed the space $H(\U)$ of all homomorphisms of $\U$ and gave a fine description of it as follows:
$H(\U)$ is homeomorphic to a quotient space\footnote{
  Detailed equivalence relations in the quotient space are not described here because they require preparation that is not needed in this note. See \cite{CS} for details.
}
 obtained from
  $[1,2]\times \beta \N \times (0,\infty)$
  with one point added, where $\beta \N$ denotes the Stone-\v Cech compactification of natural numbers.
Also, by considering $H(\U^{\ast})$, he gave a description of the Samuel-Smirnov compactification of $\HH$ (cf. \cite{FJCS}, \cite{Woods}).

Inspired by his work, we study the space $H(\SO)$ of all homomorphisms of the vector lattice of slowly oscillating functions on $\HH$.
In contrast to the case of homomorphisms of  uniformly continuous functions, it is shown that a homomorphism in $H(\SO)$ which maps the unit to zero must be the zero-homomorphism
(Proposition \ref{3-9}).
Consequently, we show that the space $H(\SO)$ without the zero-homomorphism is homeomorphic to $\HH\times (0, \infty)$.
By describing a neighborhood base of the zero-homomorphism, we show that $H(\SO)$ is homeomorphic to the space $\HH\times (0, \infty)$ with one point added
(Theorem \ref{3-10}).

\section{Preliminaries}
Throughout this note, $\HH$ denotes the half-line $[0, \infty)$ with the metric given by the absolute value $|x-y|,\ x,y\in \HH$, and $\N$ denotes the space of natural numbers with the subspace metric.
Also, $X=(X, d_X)$ is assumed to be a metric space.

Let $\LL\subset C(X)$ be a unital vector lattice,
that is, $\LL$ contains the unit $\mathbf{1}: X\to \R$.
The sublattice of all bounded functions of $\LL$ is denoted by $\LL^{\ast}$.
A function $\phi :\LL \to \R$ is called a {\it homomorphism} if it 
is a linear map preserving joins and meets, that is, $\phi$ satisfies
\begin{enumerate}
\item[(i)]  $\phi(f\vee g)=\phi(f)\vee \phi(g)$,~
 $\phi(f\wedge g)=\phi(f)\wedge \phi(g)$, and
\item [(ii)] $\phi (\lambda \cdot f+ \mu \cdot g)= \lambda\cdot\phi (f) +\mu \cdot\phi(g)$
\end{enumerate}
for all $f,g\in \LL,~\lambda, \mu \in \R$.
Note that (i) and (ii) implies
\begin{enumerate}
\item[(iii)] $\phi(|f|)=|\phi (f)|$ for all $f\in \LL$.
\end{enumerate}
Indeed, the formulation 
\[|f|=f\vee \mathbf{0}-f\wedge \mathbf{0}\]
implies that
\begin{align*}
  \phi(|f|)
  &=\phi(f)\vee \phi(\mathbf{0})-\phi(f)\wedge\phi(\mathbf{0})\\
  &=\phi(f)\vee \mathbf{0}-\phi(f)\wedge \mathbf{0}\\
  &=|\phi(f)|.
\end{align*}
Recall that join and meet induce a partial order $\leq$ on $H(\LL)$,
that is,
\[f\leq g \Longleftrightarrow f=f\wedge g\]
or equivalently,
\[f\leq g \Longleftrightarrow g=f\vee g.\]
Then (i) implies that
\begin{enumerate}
\item[(iv)] $\phi(f)\leq \phi(g)$ whenever $f\leq g$.
\end{enumerate}
Besides, (iii) implies that a homomorphism $\phi$ is {\it positive}, that is,
\begin{enumerate}
\item[(v)] $\phi(f)\geq 0$ whenever $f\in \LL$ satisfies $f\geq 0$.
\end{enumerate}
The set of all homomorphisms $\phi :\LL\to \R$ is denoted by $H(\LL)$.
Note that $H(\LL)$ is a subset of $\R^{\LL}$.
We consider the topology on $H(\LL)$ inherited from $\R^{\LL}$.
Hence, a basic neighborhood of $\phi \in H(\LL)$
is given by
\[V(\phi; f_1, \dots, f_n; \varepsilon)
=\{\varphi\in H(\LL) : |\varphi(f_i)-\phi(f_i)|<\varepsilon,~\forall i=1,\dots, n \},
\]
where $\varepsilon >0$ and $f_i \in \LL$, $i=1,\dots, n$.
Put
\[K(\LL)
=\{\phi \in H(\LL): \phi (\mathbf{1})=1\}.\]
Then it is easy to see that $K(\LL)\subset H(\LL)$, and $H(\LL)$ and $K (\LL)$ are closed subspaces of $\R^{\LL}$.
In particular, $H(\LL^{\ast})$ and $K (\LL^{\ast})$ are compact spaces.
Indeed,
they are closed subspaces of the Cartesian product
\[\prod_{f\in \LL^{\ast}}\bigl[\inf f, \sup f\bigr].\]

For each $x\in X$, let $\delta_x : \LL \to \R$ be the evaluation homomorphism defined by
$\delta_x (f)=f(x)$ for every $f\in \LL$.
We note that $\delta_x (\mathbf{1})=1$ for every $x\in X$,
 i.e., $\delta_{x}\in K(\LL)$.
Then define
\[\delta : X\to K(\LL)\]
by $\delta(x)=\delta_x$ for each $x\in X$.
When we treat $\LL^{\ast}$, consider the map
\[e_{\LL^{\ast}}: X\to \prod_{f\in \LL^{\ast}}\bigl[\inf f, \sup f\bigr],\]
defined by $e_{\LL^{\ast}}(x)=(f(x))_{f\in \LL^{\ast}}$ for every $x\in X$.
One should note that the two maps $\delta : X\to K(\LL^{\ast})\subset \R^{\LL}$
and $e_{\LL^{\ast}}: X\to \prod_{f\in \LL^{\ast}}[\inf f, \sup f]\subset \R^{\LL}$
are essentially the same correspondence.

A unital vector lattice $\LL\subset C(X)$ is said to {\it separate points and closed sets} in $X$ provided that, for each closed set $F\subset X$ and each point $p\in X\setminus F$,
there exists $f\in \LL$ such that $f(p)\not\in \mbox{cl}_{\mathbb{R}}\, f(F)$.

The following is a fundamental fact concerning $K(\LL)$
(see \cite[pp. 129--130]{Garrido-Jaramillo}, \cite[1.7 (j)]{PW}).

\begin{proposition}\label{2-1}
If $\LL$ separates points and closed sets in $X$, then $\delta : X\to K(\LL)$ is a dense topological embedding.
\end{proposition}

Though $K(\LL)$ is not compact in general,
it can be considered as a realcompactification of $X$ by Proposition \ref{2-1}.
See \cite{Garrido-Jaramillo} for more information about realcompactifications.

Let $\U(X)$ denote the lattice of all real-valued uniformly continuous functions on $X$.
We write $\U$ (resp. $\U^{\ast}$) instead of $\U(\HH)$ (resp. $\U(\HH)^{\ast}$) for notational simplicity.
The family $\U^{\ast}(X)$ has a ring structure with respect to $\R$,
but $\U(X)$ does not.
Therefore, when considering unbounded vector lattices, we need to consider lattice homomorphisms instead of ring homomorphisms.

Let $\alpha X$ and $\gamma X$ be compactifications of $X$.
We say $\alpha X \succeq \gamma X$ provided that
there is a continuous map $f:\alpha X \to \gamma X$ such that $f|_{X}=\mbox{id}_X$.
If $\alpha X \preceq  \gamma X$ and $\alpha X \succeq \gamma X$
then we say that $\alpha X$ and $\gamma X$ are {\it equivalent compactifications} of $X$.
Of course, two equivalent compactifications of $X$ are homeomorphic.

It is easy to check that $\U^{\ast}(X)$ contains all constant maps, separates points from closed sets,
and is a closed subring of $C^{\ast} (X)$ with respect to the sup-metric,
i.e., $\U^{\ast}(X)$ is a {\it complete ring on functions} (see \cite[3.12.22(e)]{Eng}).
Hence, $\U^{\ast}(X)$ uniquely determines a compactification $u X$ of $X$
(see \cite[3.12.22 (e)]{Eng}, \cite[4.5]{PW}),
which is called the {\it Samuel-Smirnov compactification} of $X$
(see \cite{CS}, \cite{Woods}).
We note that $uX$ is equivalent to
 $K(\U^{\ast}(X))= \mbox{cl}_{\R^{\U^{\ast}(X)}}\delta (X)$
because of the equivalence of two maps $\delta :X\to K(\U^{\ast}(X))$
and $e_{\U^{\ast}(X)}: X\to \prod_{f\in \U^{\ast}(X)}\bigl[\inf f, \sup f\bigr]$.

Let $(X, d_X)$ be a metric space and let $B_{d_X}(x,r)$
 be the closed ball of radius $r$ centered at $x\in X$.
A metric $d_X$ on $X$ is called {\it proper} if $B_{d_X}(x, r)$ is compact
 for every $x\in X$ and $r>0$.

Let $(X,d_X )$ and $(Y,d_Y)$ be proper metric spaces.
A map $f: X\to Y$ is said to be {\it slowly oscillating}
 provided that, given $R>0$
  and $\varepsilon>0$, there exists a compact subset $K\subset X$
  such that 
\[\diam_{d_Y} f(B_{d_X}(x,R))<\varepsilon\]
for every $x\in X\setminus K$,
where $\diam_{d} A=\sup \{ d(x,y): x,y\in A\} $.
 Let $\SO(X)$ denote the lattice of
  all real-valued slowly oscillating continuous functions on a proper metric space $X$.
 The sublattice of all bounded functions of $\SO(X)$ is denoted by $\SO(X)^{\ast}$.
When $X=\HH$ we just write $\SO$ (resp. $\SO^{\ast}$) instead of $\SO(\HH)$ (resp. $\SO(\HH)^{\ast}$) for notational simplicity.
It is easy to check that $\SO^{\ast}(X)$ is
 a closed subring of $C^{\ast}(X)$ with respect to the sup-metric, 
namely, a complete ring on functions.
Hence, $\SO^{\ast}(X)$ uniquely determines a compactification $hX$ of $X$,
 which is called the {\it Higson compactification} of $X$.
The remainder $\nu X=hX \setminus X$ is called the {\it Higson corona} of $X$
 (cf. \cite{Roe}, \cite{Keesling}).
We note that $\nu X$ is compact and that $hX$ and $K(\SO^{\ast}(X))$ are equivalent compactifications of $X$.
\par

\begin{proposition}\label{2-2}
  If $(X, d)$ is a proper metric space, then $\SO(X)\subset \U(X)$.
\end{proposition}
\begin{proof}
  Let $f\in \SO(X)$.
  Given $\varepsilon>0$, there exists a compact subset $K\subset X$ such that
  $\diam f(B(x, 1))<\varepsilon$ whenever $x\in X\setminus K$.
  Put $K'=\mbox{cl}\, B(K,1)$.
  Since $X$ is a proper metric space, $K'$ is compact.
  Consider a family 
  \[\mathscr{U}=\{ f^{-1}(B(f(x), \varepsilon/2)): x\in K'\}.\]
  Since $K'\subset \bigcup \mathscr{U}$,
   we can take a Lebesgue number $\delta_0 >0$ of $\mathscr{U}$,
   that is, every $\delta_0$-neighborhood of $x\in K'$ is contained in some element of $\mathscr{U}$.
   Let $\delta=\min\{\delta_0, 1\}$.
  Then $d(x,y)<\delta$ implies that $x,y\in K'$ or $x,y\in X\setminus K$.
  Hence, $d(f(x), f(y))<\varepsilon$ whenever $d(x,y)<\delta$.
\end{proof}


\section{Homomorphisms of the lattice of slowly oscillating functions
 on the half-line}

We note that $f:\HH \to \R$ is a slowly oscillating function ($f\in \SO$)
 if and only if
 for every $R>0$ and $\varepsilon >0$ there exists $M>0$ such that
\[\mbox{diam}\, f([x, x+R])<\varepsilon \ \mbox{for every }x>M.\]

\indent
Let $\tau : \HH \to \R$ be the map defined by \[\tau (x)=x+1\] for every $x\in \HH$.
One should note that $\tau^{\alpha}\in \SO$ for every $0<\alpha<1$.
\par
For each $f\in \SO$, we consider the map $f_{\ast}: H(\SO)\to \R$
 defined by 
 \[f_{\ast}(\phi)=\phi (f) \]
 for every $\phi\in H(\SO)$. 
 Recall that a basic neighborhood of $\phi \in H(\SO)$ is of the form
 \[V(\phi; f_1, \dots, f_n; \varepsilon)
 =\{\varphi\in H(\SO) : |\varphi(f_i)-\phi(f_i)|<\varepsilon,~\forall i=1,\dots, n \},
 \]
 where $\varepsilon >0$ and $f_i \in \SO$, $i=1,\dots, n$.
Now it is easy to see that $f_{\ast}$ is continuous.

\begin{proposition}\label{3-1}
  $K(\SO )=\delta(\HH)$.
\end{proposition}

\begin{proof}
It is obvious that $\delta(\HH)\subset K(\SO )$.
We shall show that $K(\SO )\subset \delta(\HH)$.
Let $\phi\in K(\SO)$.
Note that $\delta(\HH)$ is dense in $K(\SO)$ by Proposition \ref{2-1}.
Thus we can take a net $(x_{\alpha})_{\alpha}$ in $\HH$ such that $(\delta_{x_{\alpha}})_{\alpha}$ converges to $\phi$.
For each $f\in \SO$,
 the net $(f_{\ast}(\delta_{x_{\alpha}}))_{\alpha}=(\delta_{x_{\alpha}}(f))_{\alpha}=(f(x_{\alpha}))_{\alpha}$ converges to $f_{\ast}(\phi)=\phi(f)$ because $f_{\ast}$ is continuous,
 that is,
 \[\phi (f)=\lim_{\alpha} f(x_{\alpha}).\]
Taking $f=\sqrt{\tau}\in \SO$, we have
\[ \phi(\sqrt{\tau})= \lim_{\alpha} \sqrt{x_{\alpha}+1}. \]
Put $x_{\phi}=(\phi (\sqrt{\tau}))^2 -1$.
Then we have $x_{\phi}=\lim_{\alpha}x_{\alpha}$.
Hence, we conclude that $\phi =\delta_{x_{\phi}}\in \delta(\HH)$,
 i.e., $K(\SO )\subset \delta(\HH)$.
\end{proof}

\begin{corollary}\label{3-2}
 For each $\phi\in H(\SO)$, $\phi(\mathbf{1} )>0$  if and only if 
 there exist $x_{\phi}\in \HH$ and $c>0$ such that
 $\phi =c\cdot\delta_{x_{\phi}}$. 
In particular, if $\phi(\mathbf{1} )>0$
 then the point $x_{\phi}\in\HH$ is uniquely determined.
\end{corollary}

\begin{proof}
  If $\phi (\mathbf{1})>0$ then $\phi(\mathbf{1})^{-1}\cdot\phi \in K(\SO)$.
By Proposition \ref{3-1}, there exists $x_{\phi}\in \HH$
 such that $\phi(\mathbf{1})^{-1}\cdot\phi =\delta_{x_{\phi}}$,
 i.e., $\phi =\phi(\mathbf{1})\cdot\delta_{x_{\phi}}$.
The reverse implication is trivial.
\par
Suppose that $\phi(\mathbf{1})>0$
 and $\phi=\phi(\mathbf{1})\cdot\delta_{s}=\phi(\mathbf{1})\cdot\delta_{t}$
 for some $s, t\in \HH$ then the equation
 $\phi(\tau)=\phi(\mathbf{1})\cdot(s+1)=\phi(\mathbf{1})\cdot(t+1)$
 implies that $s=t$.
\end{proof}

The following two lemmas are modifications of those stated in \cite[p. 418]{CS}.

\begin{lemma}\label{3-3}
Let $f \in \SO$ be a map such that $f\geq \mathbf{1}$.
   If there exits $\phi\in H(\SO)$ such that
   $\phi(f)=1$ and $\phi(\mathbf{1})=0$,
   then $\phi$ is contained in the closure of $\{ f(n)^{-1} \cdot \delta_n : n \in \N\}$
   in $H(\SO)$.
 \end{lemma}
 
 \begin{proof}
   Suppose that there exists $\phi\in H(\SO)$ such that
   $\phi(f)=1$ and $\phi(\mathbf{1})=0$ but which is not contained in the closure of
    $\{ f(n)^{-1}\cdot \delta_n : n \in \N\}$
    in $H(\SO)$.
   Then there exist $\varepsilon >0$ and $g_1,\dots, g_k \in\SO$
   such that $f(n)^{-1} \cdot \delta_n \not\in V(\phi ; g_1, \dots, g_k; \varepsilon) $
   for every $n\in \N$.
So, for each $n\in \N$, 
 there exists $i\in \{1,\dots, k\}$ such that
   \[\left| \phi (g_i )-f(n)^{-1} \cdot g_i(n)\right|\geq \varepsilon. \]
Hence, we have
   \[\bigvee_{i=1}^{k} \left| \phi (g_i)\cdot f(n) -g_i (n) \right| \geq \varepsilon \cdot f(n)\]
   for every $n\in \N$.
Let $c_i =\phi(g_i)$ for each $i=1,\dots, k$.
Put
   \[h=0\wedge \left(\bigvee_{i=1}^{k} \left|c_i \cdot f -g_i \right|-\varepsilon \cdot f \right).\]
Then $h\in \SO\subset \U$ and $h(n)=0$ for every $n\in \N$.
It follows from uniformity that $h$ is a bounded function.
So, there exists $c>0$ such that $|h|\leq c\cdot \mathbf{1}$.
Thus, we have 
$|\phi(h)|=\phi(|h|)\leq c\cdot \phi (\mathbf{1})=0$, i.e., $\phi(h)=0$.
 We note that 
 \[ \bigvee_{i=1}^{k} \left| c_i \cdot  f -g_i \right|
  \geq h+\varepsilon \cdot f\]
 and
 \[\phi\left(\bigvee_{i=1}^{k} \left|c_i \cdot  f  -g_i \right|\right)
 =\bigvee_{i=1}^{k} \left|c_i \cdot \phi(f)  -\phi(g_i) \right|
 =0.\]
 However, we have $\phi(h+\varepsilon \cdot f )=\phi(h)+\varepsilon \cdot  \phi(f)=\varepsilon>0$, a contradiction.
 \end{proof}

 Let $\F$ be an ultrafilter on $\N_{0}=\N\cup\{0 \}$.
 Then we define the operation $\lim_{\F(n)}$ by
 \[\lim_{\F(n)} f(n)=\bigcap_{F\in \F} \mbox{cl}\left\{ f(n): n\in F\right\}  \]
 for every $f\in C(\HH)$ (cf. \cite{CS}).
If $f\in C(\HH)$ is a map such that
   $\displaystyle\lim_{\F(n)} f(n)\neq \emptyset$
   then the set $\displaystyle\lim_{\F(n)} f(n)$ consists of a single point
   since $\F$ is an ultrafilter.
\par
Recall that the Stone-\v{C}ech compactification $\beta\N_{0}$
 of $\N_{0}$ can be considered as the space of all ultrafilters on $\N_{0}$.

 \begin{lemma}\label{3-4}
Let $f \in \SO$ be a map such that $f\geq \mathbf{1}$.
Suppose that there exists a homomorphism  $\phi\in H(\SO)$ such that
    $\phi( f)=1$ and $\phi(\mathbf{1})=0$.    
Then there exists a free ultrafilter $\F$ such that
 the function $\phi_{\F}^{f} : \SO \to \R$ defined by
\[\phi_{\F}^{f}(g) =\lim_{\F(n)} \frac{g(n)}{ f(n)}  , ~~~(g\in \SO)\]
is a well-defined homomorphism that fulfils $\phi_{\F}^{f} =\phi$.
\end{lemma}

\begin{proof}
  Suppose that there exists a homomorphism  $\phi\in H(\SO)$ such that
    $\phi( f)=1$ and $\phi(\mathbf{1})=0$ for some
    $ f\in \SO$ with $f\geq \mathbf{1}$.
  For each neighborhood $V$ of $\phi$ in $H(\SO)$,
  let 
  $N_{V}=\left\{ n\in \N :  f(n)^{-1} \cdot \delta_n \in V  \right\}$.
  Then $N_{V}\neq \emptyset$ by Lemma \ref{3-3}.
  Put \[\G =\{N_{V} : V~\mbox{is a neighborhood of}~\phi\}.\]
Then $\G$ becomes a filter on $\N$.
Let $\F$ be an ultrafilter on $\N$ refining $\G$.
We note that $\F$ must be a free ultrafilter since $\phi(\mathbf{1})=0$.
Indeed, if $\G$ is a fixed ultrafilter, say $\lim \G=n_0$, then $\phi = f(n_0)^{-1}\delta_{n_0}$.
Hence, we have $\phi(\mathbf{1})=f(n_0)^{-1}\neq 0$, a contradiction.
\par

Given $\varepsilon>0$ and $g\in \SO$, we consider a neighborhood
 $V_{\varepsilon}=V(\phi;g;\varepsilon)$ of $\phi$ in $H(\SO)$, that is,
\[V_{\varepsilon}
=\{\varphi\in H(\SO) : |\varphi(g)-\phi(g)|<\varepsilon \}.\]
Since $ \left\{ 
    n\in \N :  f(n)^{-1} \cdot \delta_n \in V_{\varepsilon}
   \right\} \in \G\subset \F$,
  we have
   \[\phi_{\F}^{f}(g)
   =\bigcap_{F\in \F} \mbox{cl}\left\{ \frac{g(n)}{ f(n)} : n\in F\right\}
   \subset
   \bigcap_{G\in \G} \mbox{cl}\left\{ \frac{g(n)}{ f(n)} : n\in G\right\}\subset
    B(\phi(g), \varepsilon).\]
Then  $\phi_{\F}^{f}(g)$ is not empty by the compactness of $B(\phi(g), \varepsilon)$
 and it is uniquely determined because $\F$ is an ultrafilter.
Since $\varepsilon$ is arbitrary, it follows that  $\phi_{\F}^{f}(g)=\phi(g)$,
 that is, $\phi_{\F}^{f}$ is a well-defined homomorphism that fulfils $\phi_{\F}^{f} =\phi$.
\end{proof}

The following lemma is the key to this note, as it implies that a homomorphism in $H(\SO)$ which maps the unit to zero must be the zero-homomorphism (Proposition \ref{3-9}). Using this fact, we will derive our main result (Theorem \ref{3-10}).

 \begin{lemma}[Vanishing Criterion]\label{3-5}
  Let $\phi \in H(\SO)$.
If there are two maps $f, g \in \SO$ such that $\mathbf{1}\leq  f\leq g$
and
$\displaystyle\lim_{n\to \infty}f(n)^{-1} \cdot g(n)=\infty$,
then the condition $\phi(\mathbf{1})=0$ implies that $\phi (f)=0$.
\end{lemma}
    
\begin{proof}
  Let $f, g \in \SO$ be such that $\mathbf{1}\leq  f\leq g$
  and
  $\lim_{n\to \infty}f(n)^{-1} \cdot g(n)=\infty$.
  Let $\phi\in H(\SO)$ be such that $\phi (\mathbf{1})=0$.
  Suppose that $\phi(f)\neq 0$.
  Replacing $\phi$ by $\phi(f)^{-1} \cdot \phi$,
   we may assume that $\phi (f)=1$.
  Then, by Lemma \ref{3-4}, there exists a free ultrafilter $\F$ such that
  $\phi=\phi_{\F}^{f}$.
  However,  since $\F$ is a free ultrafilter, we have
 \[\phi (g)=\phi_{\F}^{f}(g) =\lim_{\F(n)} \frac{g(n)}{ f(n)} =\infty , \]
 a contradiction.
\end{proof}

\begin{definition}\label{3-6}
  A sequence $\aaa=(a_n)\subset \N$ is called a {\it strictly increasing sequence} provided that $a_n <a_{n+1}$ for every $n\in \N$.
  Note that if $\aaa=(a_n)$ is a strictly increasing sequence then $\displaystyle\lim_{n\to \infty}a_n =\infty$
   since $\aaa \subset \N$.
  \par
  Let $\aaa$ be a strictly increasing sequence.
  Let $\eta_{\aaa}^{0}=\tau: \HH \to \R$. 
  Suppose that $\eta_{\aaa}^{n-1}$ has been defined
   for $n\geq 1$.
  Then we define $\eta_{\aaa}^n : \HH \to \R$ by
    \[\eta_{\aaa}^{n}(x)=\left\{
    \begin{array}{ll}
      \eta_{\aaa}^{n-1}(x),& 0\leq x< a_{n} ,\\[1mm]
      \eta_{\aaa}^{n-1}(a_{n})+\frac{1}{n}\left(x-a_{n}\right), & a_{n}\leq x,
    \end{array}
    \right.
    \]
    for every $x\in \HH$ (see Figure \ref{fig1}).
    \begin{figure}[h]
      \centering
      \begin{tikzpicture}
      \draw(0.1,0.1)node[below left]{$0$};
      \draw(0,0.5)node[left]{$1$};
      \coordinate[label=left:](A)at(0,0.5);
      \coordinate[label=right:](B)at(2,2.5);
      \coordinate[label=above:](C)at(4,3.5);
      \foreach\P in{A,B,C}\fill[black](\P)circle(0.06);  
      \draw[-Stealth](0,0)--(8,0);
      \draw[-Stealth](0,0)--(0,6);
      \draw[densely dotted](0,0.5)--(8,0.5);
      \draw[dashed](0,0.5)--(5.5,6);
      \draw(4.5,4.8)node[above left]{$\eta_{\aaa}^{0} =\tau$};%
      \draw(2,0)node[below]{$a_1$};
      \draw(4,0)node[below]{$a_2$};
      \draw[densely dotted](2,0)--(2,2.5);
      \draw[densely dotted](4,0)--(4,4.5);
      \draw[thin](0,0.5)--(2,2.5)--(8,5.5);
      \draw(6,4.5)node[above left]{$\eta_{\aaa}^{1}$};%
      \draw[ultra thick](0,0.5)--(2,2.5)--(4,3.5)--(8,4.83);
      \draw(7.4,4.5)node[below left]{$\eta_{\aaa}^{2}$};%
      \end{tikzpicture}
      \caption{The graphs of $\eta_{\aaa}^{0}$, $\eta_{\aaa}^{1}$ and $\eta_{\aaa}^{2}$.}\label{fig1}
    \end{figure}
    Note that $\eta_{\aaa}^{n-1}\geq \eta_{\aaa}^{n} \geq 1$ for every $n\in \N$.
    \par
    We define $\eta_{\aaa}:\HH \to \R$ by
      \[\eta_{\aaa} (x)= \lim_{n\to \infty} \eta_{\aaa}^{n} (x)\]
      for every $x\in \HH$.
    We note that if $x\leq a_n$ then
    \[\eta_{\aaa}(x)=\eta_{\aaa}^{n}(x)=\eta_{\aaa}^{n-1}(x).\]
    It is easy to see that $\eta_{\aaa} :\HH\to \R$ is a well-defined slowly oscillating continuous function  such that $\eta_{\aaa}\geq 1$.
    We call $\eta_{\aaa}$ the {\it slowly oscillating function with respect to $\aaa$}.
    \end{definition}

    \begin{proposition}\label{3-7}
      For each $f\in \SO$, there exists a strictly increasing sequence
       $\aaa\subset\N$ and
      $L>0$ such that $|f|\leq L\cdot \eta_{\aaa}$.
    \end{proposition}
    
    \begin{proof}
      Since $f\in \SO$, we can take a strictly increasing sequence
       $\aaa=(a_n)\subset\N$
      such that 
      \begin{enumerate}
        \item $\diam f(B(x, 1))<(n+1)^{-4}$ for every $x\geq a_n$.
      \end{enumerate}
  Let $L=1+\sup\{ |f(x)|: x\leq a_1\}$.
  Then we have 
    \[|f(x)|+1 \leq L\leq L\cdot \tau(x) =L\cdot \eta_{\aaa}^{0}(x)\]
      for every $x\leq a_{1}$.
  Suppose that we have shown that
      \begin{enumerate}
        \item[$(2)_{n}$] $|f(x)|+n^{-2}\leq L\cdot \eta_{\aaa}^{n-1}(x)$ for every $x\leq a_{n}$.
      \end{enumerate}
  If $x\leq a_{n}$ then $(2)_{n+1}$ follows from $(2)_n$ since $|f(x)|+(n+1)^{-2}\leq|f(x)|+n^{-2}$
   and $\eta_{\aaa}^{n}(x) =\eta_{\aaa}^{n-1}(x)$.
  Now suppose that $a_{n}\leq x\leq a_{n+1}$.
  Then we have
      \begin{align*}
        |f(x)|+\frac{1}{(n+1)^2}
        &\leq |f(a_{n})|+\frac{x-a_n}{(n+1)^4} +\frac{1}{(n+1)^4}+\frac{1}{(n+1)^2}\qquad (\mbox{by }(1))\\
        &< |f(a_{n})|+\frac{x-a_n}{n+1} +\frac{1}{n^2}\\
        &\leq L\cdot \eta_{\aaa}^{n}(a_{n})+\frac{x-a_n}{n+1} \qquad (\mbox{by }(2)_n )\\
        &\leq L\cdot \eta_{\aaa}^{n-1}(a_{n})+\frac{x-a_n}{n}
        \qquad (\because~\eta_{\aaa}^{n}(a_{n})=\eta_{\aaa}^{n-1}(a_{n}))\\
        &\leq L\cdot \left( \eta_{\aaa}^{n-1}(a_{n})+\frac{x-a_n}{n} \right)
         \qquad (\because~ L\geq 1)\\
        &=L\cdot \eta_{\aaa}^{n}(x).
      \end{align*}
  Thus $(2)_{n+1}$ holds.
  Consequently, we have $|f|\leq L\cdot \eta_{\aaa}$
   since $\displaystyle\lim_{n\to \infty}\eta_{\aaa}^{n} =\eta_{\aaa}$
    and $\displaystyle\lim_{n\to \infty} a_n =\infty$.
    \end{proof}

\begin{proposition}\label{3-8}
      For each strictly increasing sequence $\aaa\subset \N$,
       there exists a strictly increasing sequence $\bbb\subset \N$
       such that
       $\eta_{\aaa}\leq \eta_{\bbb}$ and
      $\displaystyle\lim_{n\to \infty}\eta_{\aaa}(n)^{-1} \cdot \eta_{\bbb}(n)=\infty$.
\end{proposition}   
    
    \begin{proof}
  Let $\aaa=(a_n)\subset \N$ be a strictly increasing sequence.
  Let $\bbb=(b_n)$ be a strictly increasing sequence such that
  \begin{enumerate}
    \item[(1)] $b_0 =a_1$ and
    \item[(2)] $b_n \geq n^2 \cdot \eta_{\aaa}(b_{n-1})+b_{n-1} +a_{(n+1)^3}$ for each $n\in \N$.
  \end{enumerate}
  We shall show that $\eta_{\bbb}(x)\geq n\cdot \eta_{\aaa}(x)$ for every $x\in [b_{n}, b_{n+1}]$.
  \par
  Since $b_i > a_i$ for $i=1,2$, we have $\eta_{\bbb}(x)\geq 1\cdot\eta_{\aaa}(x)$
      for every $x\in [b_1, b_2]$.
  Suppose that we have shown that
  \begin{enumerate}
    \item[(3)] $\eta_{\bbb}(x)\geq (n-1)\cdot\eta_{\aaa}(x)$
    whenever $x\in [b_{n-1}, b_{n}]$ for $n\geq 2$.
  \end{enumerate}
  Let $x\in [b_n, b_{n+1}]$. We write $x=b_{n-1}+t$, $t> 0$
       for some technical reason.
  Then we have
      \begin{equation*}
        \frac{t}{n^2}>\eta_{\aaa}(b_{n-1}) \tag{4}
      \end{equation*}
  Indeed, since $x=b_{n-1}+t \geq b_n$, we have
  \begin{align*}
    t&\geq b_n -b_{n-1}\\
    &> n^2 \cdot \eta_{\aaa}(b_{n-1})+a_{(n+1)^3}
    \qquad (\mbox{by }(2))\\
    &>n^2 \cdot \eta_{\aaa}(b_{n-1}).
  \end{align*}
  Then we have
      \begin{align*}
        \eta_{\bbb}(x)
        &=\eta^{n}_{\bbb}(x)
        =\eta_{\bbb}^{n-1}(b_n) +\frac{1}{n}(x-b_n)\\
        &=\eta_{\bbb}^{n-2}(b_{n-1})+\frac{1}{n-1}(b_{n}-b_{n-1})+\frac{1}{n}(x-b_n)\\
        &\geq \eta_{\bbb}^{n-2}(b_{n-1})+\frac{1}{n}(b_{n}-b_{n-1}+x-b_n)\\
        &= \eta_{\bbb}^{n-2}(b_{n-1})+\frac{1}{n}(x-b_{n-1})\\
        &=\eta^{n-2}_{\bbb}(b_{n-1})+\frac{t}{n}\\
        &=\eta_{\bbb}(b_{n-1})+\frac{t}{n}\\
        &\geq (n-1)\cdot \eta_{\aaa}(b_{n-1})+\frac{t}{n}.\tag{5}
      \end{align*}
  The last inequality follows from $(3)$.
      Since $b_{n-1} > a_{n^3}$, there exists $k\geq n^3$ such that
      $a_{k} \leq b_{n-1}<a_{k+1}$.
      Then we have
      \begin{align*}
        \eta_{\aaa}(x)
        &\leq \eta^{k}_{\aaa}(x)
        =\eta^{k-1}_{\aaa}(a_k) +\frac{1}{k}(x-a_{k})\\
        &=\eta^{k-1}_{\aaa}(a_k) + \frac{1}{k}(b_{n-1}-a_{k})+\frac{1}{k}(x-b_{n-1})\\
        &=\eta^{k}_{\aaa}(b_{n-1})+\frac{1}{k}(x-b_{n-1})\\
        &=\eta^{k}_{\aaa}(b_{n-1})+\frac{t}{k}\\
        &\leq \eta_{\aaa}(b_{n-1})+\frac{t}{n^3}.\tag{6}
      \end{align*}
  Hence, we have
  \begin{align*}
    \eta_{\bbb}(x)-n\cdot \eta_{\aaa}(x)
    &\geq (n-1)\cdot \eta_{\aaa}(b_{n-1})+\frac{t}{n} -n\cdot \eta_{\aaa}(x) \qquad (\mbox{by }(5))\\
    &\geq (n-1)\cdot \eta_{\aaa}(b_{n-1})+\frac{t}{n} 
    -n\cdot \left(\eta_{\aaa}(b_{n-1})+\frac{t}{n^3}\right)
    \qquad  (\mbox{by }(6))\\
    &=(n-1)\cdot\frac{t}{n^2}-\eta_{\aaa}(b_{n-1})\\
    &> (n-1)\cdot \eta_{\aaa}(b_{n-1})-\eta_{\aaa}(b_{n-1}) \qquad  (\mbox{by }(4))\\
    &=(n-2)\cdot \eta_{\aaa}(b_{n-1})\\
    &\geq 0.
  \end{align*}
  Thus we conclude that
  \[\lim_{n\to\infty}\frac{\eta_{\bbb}(n)}{\eta_{\aaa}(n)}
  \geq \lim_{n\to\infty}n =\infty.\]
  \end{proof}

  \begin{proposition}\label{3-9}
    If $\phi\in H(\SO)$ satisfies $\phi(\mathbf{1})=0$
    then $\phi =\mathbf{0}$.
  \end{proposition}
  
  \begin{proof}
    Let $f\in \SO$.
    Suppose that $\phi(\mathbf{1})=0$.
    By Proposition \ref{3-7}, there exists a strictly increasing sequence $\aaa\in \N$ and $L>0$
    such that $|f|\leq L\cdot \eta_{\aaa}$.
  Recall that $\eta_{\aaa}\geq \mathbf{1}$.
  By Proposition \ref{3-8},
   there exists a strictly increasing sequence $\bbb\in \N$ so that
   $\mathbf{1} \leq \eta_{\aaa}\leq \eta_{\bbb}$ and
  \[\lim_{n\to \infty}\frac{\eta_{\bbb}(n)}{\eta_{\aaa}(n)}=\infty.\]
  Thus, 
  the condition $\phi(\mathbf{1})=0$
    implies that $\phi (\eta_{\aaa})=0$ by the Vanishing Criterion (Lemma \ref{3-5}).
    Hence, we have
    \[|\phi(f)|=\phi (|f|)\leq L\cdot \phi(\eta_{\aaa})=0,\]
    i.e., $\phi(f)=0$.
  Since $f$ can be taken arbitrary,
   $\phi $ must be the zero-homomorphism.
  \end{proof}

  By Proposition \ref{3-9}, 
  it follows that the structure of $H(\SO)$ is very simple
  in contrast to the case of uniformly continuous functions \cite{CS} (see also \cite{FJCS}).

  \begin{theorem}\label{3-10}
  The space $H(\SO)$ of all homomorphisms of the vector lattice of all slowly oscillating functions on the half-line $\HH$ is homeomorphic to the space $\left(\HH\times(0,\infty)\right)\cup \{\mathbf{0}\}$ 
   where a neighborhood base of the point $\mathbf{0}$ consists of sets of the form:
   \[\{(x,y)\in \HH\times(0,\infty) : y\leq \varepsilon \cdot\eta_{\aaa}(x)^{-1}\}\cup \{\mathbf{0}\}\]
   for some $\varepsilon >0$ and the slowly oscillating function $\eta_{\aaa}$ with respect to
   some strictly increasing sequence $\aaa$.
  \end{theorem}
  
  \begin{proof}
  By Proposition \ref{3-9}, every non-zero homomorphism $\phi\in H(\SO)$
   satisfies $\phi(\mathbf{1})>0$ since $\phi$ is positive.
  Hence, every $\phi\in H(\SO)\setminus \{\mathbf{0}\}$ is uniquely expressed as
    $\phi =\phi(\mathbf{1})\cdot \delta_{x_{\phi}}$ for some $x_{\phi}\in \HH$ by Corollary \ref{3-2}.
  Therefore, the function $\Phi: H(\SO)\setminus \{\mathbf{0}\} \to \HH\times (0, \infty)$ defined by
  \[\Phi (\phi)=(x_{\phi}, \phi(\mathbf{1}))\]
   is a well-defined bijection.
  \par
  The function $\Phi$ is continuous.
  To see this, fix $\phi \in H(\SO)$ and let $\varepsilon>0$.
  We consider following two neighborhoods of $\phi$:
  \begin{align*}
    V(\phi; \mathbf{1}; \varepsilon_1)
    &=\{\varphi \in H(\SO) : |\varphi (\mathbf{1})-\phi(\mathbf{1})|<\varepsilon_1\},\\
    V(\phi; \tau ; \varepsilon_2)
    &=\{\varphi \in H(\SO) : |\varphi (\tau)-\phi(\tau)|<\varepsilon_2\},
    \end{align*}
    where
  \begin{enumerate}
    \item[(1)] $\varepsilon_1 
    <\min\left\{ \varepsilon/2, \frac{\phi(\mathbf{1})\varepsilon}{2(x_{\phi}+1)}\right\}$ and
    \item[(2)] $\varepsilon_2 <\frac{\phi(\mathbf{1})\varepsilon}{2}-\varepsilon_1 (x_{\phi}+1)$. 
  \end{enumerate}
  Then for each $\varphi \in V(\phi; \mathbf{1}; \varepsilon_1)\cap V(\phi; \tau ; \varepsilon_2)$,
   we have
   \begin{align*}
    \phi(\mathbf{1})|x_{\phi} -x_{\varphi}|
    &=|\phi (\mathbf{1}) (x_{\phi}+1)-\phi(\mathbf{1})(x_{\varphi}+1)|\\
    &\leq |\phi (\mathbf{1}) (x_{\phi}+1)-\varphi(\mathbf{1})(x_{\phi}+1)|\\
    & \qquad + |\varphi (\mathbf{1}) (x_{\phi}+1)-\phi(\mathbf{1})(x_{\varphi}+1)|\\
    &=|\phi (\mathbf{1}) - \varphi (\mathbf{1})|(x_{\varphi}+1)
    +|\varphi(\tau) -\phi(\tau)|\\
    &<\varepsilon_1 (x_{\varphi}+1) + \varepsilon_2.
   \end{align*}
   By (2), we have
  \[  |x_{\phi} -x_{\varphi}|
  < \frac{\varepsilon_1 (x_{\phi}+1)+\varepsilon_2}{\phi(\mathbf{1})} <\varepsilon/2.\]
  Thus we have
  \begin{align*}
    d(\Phi(\phi), \Phi(\varphi))
    &\leq d((x_{\phi}, \phi(\mathbf{1})), (x_{\phi}, \varphi(\mathbf{1})))
    +d((x_{\phi}, \varphi(\mathbf{1})), (x_{\varphi}, \varphi(\mathbf{1})))\\
    &=|\varphi (\mathbf{1})-\phi(\mathbf{1})|+|x_{\varphi} -x_{\phi}|\\
    &<\varepsilon/2 + \varepsilon/2 =\varepsilon.
  \end{align*}
  \par
  Next we shall show that
   $\Phi^{-1}: \HH\times (0, \infty)\ni (x,s)\mapsto s\cdot \delta_x \in H(\SO)\setminus \{\mathbf{0}\}$ is continuous.
  Given $(x,s)\in \HH\times (0,\infty)$ and $\varepsilon >0$,
   let 
   $f\in \SO$ and 
   consider a basic neighborhood
   \[V(s\cdot\delta_x ; f; \varepsilon)
   =\{\varphi \in H(\SO) : |\varphi(f) -s\cdot f(x)|<\varepsilon\}
   \]
    of $\Phi^{-1}(x,s)=s\cdot\delta_x$.
  We take $\lambda_1 >0,$ $\lambda_2 >0$ and $\lambda_0 >0$ so that
  \begin{enumerate}
    \item[(3)] $\lambda_1 \cdot |f(x)|<\varepsilon /2$,
    \item[(4)] $ (s +\lambda_1)\cdot \lambda_2 <\varepsilon /2$,
    \item[(5)] $\lambda_0 <\min \{\lambda_1, \lambda_2\}$ and 
    \item[(6)] $|f(x)-f(y)|<\lambda_2$ whenever $|x-y|<\lambda_0$.
  \end{enumerate}
  Suppose that $(y,t)\in \HH\times (0,\infty)$ satisfies $d((x,s), (y,t))<\lambda_0$.
  Then $|x-y|<\lambda_0$ and $|s-t|<\lambda_0 \leq \lambda_1$,
  in particular, $t\leq s+\lambda_1$.
  Hence, we have
  \begin{align*}
    |\Phi^{-1}(x,s)(f)-\Phi^{-1}(y,t)(f)|
    &=|s\cdot f(x) -t\cdot f(y)|\\
    &\leq |s\cdot f(x) -t\cdot f(x)|+|t\cdot f(x) -t\cdot f(y)|\\
    &=|s-t|\cdot |f(x)|+t\cdot |f(x)-f(y)|\\
    &\leq \lambda_1 \cdot |f(x)| + (s +\lambda_1)\cdot \lambda_2\\
    &\leq \varepsilon /2 + \varepsilon /2 =\varepsilon.
  \end{align*}
  Therefore, $\Phi^{-1}(y,t) \in V(s\cdot\delta_x ; f; \varepsilon)$.
  Consequently, $\Phi$ is a homeomorphism.
  \par
  Finally, we shall consider neighborhoods of $\mathbf{0}\in H(\SO)$.
  We can take a subbase of neighborhoods of $\mathbf{0}$ in $H(\SO)$
   as a family which consists of sets of the form:
  \begin{align*}
    V(\mathbf{0}; f; \varepsilon)
    &=\{ \varphi\in H(\SO): |\varphi(f)-\mathbf{0}(f)|<\varepsilon \}\\
    &=\{ \varphi\in H(\SO): |\varphi(\mathbf{1})\cdot\delta_{x_{\varphi}}(f)|<\varepsilon \}\\
    &=\{ \varphi\in H(\SO): |\varphi(\mathbf{1})\cdot f(x_{\varphi})|<\varepsilon \}
  \end{align*}
   for some $f\in \SO$ and $\varepsilon >0$.
  
  Let $f\in \SO$ and $\varepsilon >0$.
  By Proposition \ref{3-7}, there exists $L>0$ and a strictly increasing sequence $\aaa$ such that
   $|f|\leq L\cdot \eta_{\aaa}$.
  So, if $\varphi\in V(\mathbf{0}; L\cdot\eta_{\aaa}; \varepsilon) $
   then $|\varphi(\mathbf{1})\cdot f(x_{\varphi})|<|\varphi(\mathbf{1})\cdot L\cdot\eta_{\aaa}(x_{\varphi})| <\varepsilon$,
   that is,
  \[V(\mathbf{0}; L\cdot\eta_{\aaa}; \varepsilon) \subset V(\mathbf{0}; f; \varepsilon).\]
  Since $V(\mathbf{0}; L\cdot\eta_{\aaa}; \varepsilon)
  =V(\mathbf{0}; \eta_{\aaa}; \varepsilon\cdot L^{-1})$,
  we can take a subbase of neighborhoods of $\mathbf{0}$ in $H(\SO)$
   as a family which consists of sets of the form:
   \[V(\mathbf{0}; \eta_{\aaa}; \varepsilon) 
  =\{ \varphi\in H(\SO): \varphi(\mathbf{1}) \cdot\eta_{\aaa}(x_{\varphi}) <\varepsilon\}\]
  for some $\varepsilon >0$ and a strictly increasing sequence $\aaa\subset \N$.
  Note that for any two strictly increasing sequences $\aaa=(a_n)$ and $\bbb =(b_n)$,
   if we take a strictly increasing sequence $\ccc =(c_n)$ such that $c_n \geq\max \{a_n, b_n\}$
   then we have $\eta_{\ccc}\geq \eta_{\aaa}\vee \eta_{\bbb}$, i.e.,
  \[V(\mathbf{0}; \eta_{\ccc}; \varepsilon) 
  \subset V(\mathbf{0}; \eta_{\aaa}, \eta_{\bbb}; \varepsilon) .\] 
  Thus, we can take a base of neighborhoods of $\mathbf{0}$ in $H(\SO)$
   as a family which consists of sets of the form $V(\mathbf{0}; \eta_{\aaa}; \varepsilon) $
   for some $\varepsilon >0$ and a strictly increasing sequence $\aaa\subset \N$.
  Consequently, we can take a base of neighborhoods of $\mathbf{0}$ in $\HH \times (0,\infty)\cup \{\mathbf{0}\}$
  as a family which consists of sets of the form
  \begin{align*}
    \Phi (V(\mathbf{0}; \eta_{\aaa}; \varepsilon))\cup \{\mathbf{0}\}
    &=\{(x_{\varphi}, \varphi(\mathbf{1}) )\in \HH\times(0,\infty)
    : \varphi(\mathbf{1}) \cdot\eta_{\aaa} (x_{\varphi})<\varepsilon \} \cup \{\mathbf{0}\}\\
    &=\{(x,y)\in \HH\times(0,\infty) : y\leq \varepsilon \cdot\eta_{\aaa}(x)^{-1}\}\cup \{\mathbf{0}\}
  \end{align*}
  for some $\varepsilon >0$ and an increasing sequence $\aaa\subset \N$.
  \end{proof}

  \medskip
  \begin{center}
  \textsc{Acknowledgement}
  \end{center}
  \par
  I wish to express my gratitude to the referee for the careful reading of the manuscript and valuable suggestions.


\begin{thebibliography}{99}

  \bibitem{CS}
  F. Cabello S\'anchez, 
  {\it Fine structure of the homomorphisms of the lattice of uniformly continuous functions on the line}, 
  Positivity 24 (2020), no. 2, 415--426.
  
  \bibitem{FJCS}
  F. Cabello S\'anchez and J. Cabello S\'anchez, 
  {\it Quiz your maths: Do the uniformly continuous functions on the line form a ring?}, 
  Proceedings of the American Mathematical Society 147, Issue 10 (2019), 4301--4313.

\bibitem{Eng}
R. Engelking, 
{\it General topology},
Second edition. Sigma Series in Pure Mathematics, 6. Heldermann Verlag, Berlin, 1989. 

\bibitem{Garrido-Jaramillo}
M. I. Garrido and  J. A. Jaramillo,
 {\it Homomorphisms on function lattices},
 Monatsh. Math. 141 (2004), no. 2, 127--146.

 \bibitem{Iwa2}
 Y. Iwamoto,
 {\it Indecomposable continua as Higson coronae},
 Topology Appl. 283 (2020), 107334, 16 pp.

\bibitem{Keesling}
J. Keesling,
{\it The one-dimensional \v{C}ech cohomology of the Higson compactification and its corona},
Topology Proc. 19 (1994), 129--148.

\bibitem{PW}
J. R. Porter and  R. G. Woods,
{\it Extensions and absolutes of Hausdorff spaces}, 
Springer-Verlag, New York, 1988.

\bibitem{Roe}
J. Roe,  
{\it Lectures on coarse geometry}, 
University Lecture Series, 31. American Mathematical Society, Providence, RI, 2003. 

\bibitem{Woods}
R. G. Woods, 
{\it The minimum uniform compactification of a metric spaces},
Fund. Math. 147(1995), no. 1, 39--59.

\end{thebibliography}
\end{document}